\documentclass[12pt,letterpaper]{article}
\usepackage{t1enc}
\usepackage[latin1]{inputenc}
\usepackage[english]{babel}
\usepackage{graphics}
\usepackage{color}
\usepackage{amssymb,amsmath,amsthm}

\newcommand\R{\mathbb{R}}
\newcommand\N{\mathbb{N}}
  
 \newtheorem{theo}{Theorem}

 \newtheorem{lem}{Lemma}

\begin{document}
\title{Isochronicity conditions for some real polynomial systems }

\author{Islam Boussaada
\footnote {LMRS, UMR 6085, Universit\'e de Rouen. {\small
islam.boussaada@etu.univ-rouen.fr}  } .}
\maketitle {}

\bigskip

\begin{abstract}
{\small{
This paper focuses on isochronicity of linear center perturbed by a polynomial.
Isochronicity of a linear center perturbed by a degree four and degree five polynomials is  studied, several new isochronous centers are found.
For homogeneous isochronous perturbations, a first integral and a linearizing change of coordinates are presented.

Moreover, a family of Abel polynomial systems is also considered. By investigations until degree $10$ we prove the existence of a unique isochronous center.

These results are established using a computer implementation based on Urabe theorem.
\footnote{
2000 Mathematics Subject Classification  \ 34C15, 34C25, 34C37}}}\\
 
\end{abstract}
\section{Introduction }

We consider the planar dynamical system,
\begin{equation}\label{1}
\frac{dx}{dt}=\dot x=X(x,y),\qquad \frac{dy}{dt}=\dot y=Y(x,y),
\end{equation}
where  $(x,y)$ belongs to an  open connected subset  $U\subset {\R}^2$, $X,Y\in{C}^k(U,\R)$, and  $k\geq1$.
An isolated
singular point $p\in U$ of~\eqref{1} is a center
 if and only if there exists a punctured neighborhood $V\subset U$ of $p$  such that
  every orbit in $V$ is a cycle surrounding $p$.
  
  The {\it{period annulus}} of $p$, denoted $\Gamma_p$ is the largest connected
  neighborhood covered by cycles surrounding $p$. 
The {\it{period function}}  $T:\Gamma_p\longrightarrow\R$ associate to every point
$(x,y)\in\Gamma_p$
the minimal period of the cycle $\gamma_{(x,y)}$ containing $(x,y)$.

We say that a center $p$ is {\it{isochronous}} if the period function is constant for all cycles  contained in $\Gamma_p$.
The simplest example is the linear center at the origin $O=(0,0)$ given by the system $ \dot x=-y,\; \dot y=x$.

For a cycle $\gamma\in\Gamma_p$ we denote by $C(\gamma)\subset U$ the open subset bounded by $\gamma$.  
 We say that the period function is strictly increasing (decreasing) iff 
 $T(\gamma_1)<T(\gamma_2),\;
  (T(\gamma_1)\geq T(\gamma_2)$) for all  $\gamma_1$ and $\gamma_2$ such that $C(\gamma_1)\subset C(\gamma_2)$.
 

An overview  of 
J.Chavarriga and M.Sabatini  \cite{CS} present the recent results concerning the problem of the isochronicity, see also \cite{C1,C2}.

The main purpose of this paper is the study of the Li\'enard type equation
\begin{equation}\label{L2} 
\ddot x+f(x){\dot x}^2+g(x)=0
\end{equation} 
with rational $f$ and $g$, or equivalently the study of its associated two dimentional (planar) system
\begin{equation}\label{L2P}
\left.\begin{split} \dot x &= y \\ \dot y &= -g(x) - f(x) y
 \end{split}\right\} 
   \end{equation} 

The Li\'enard type equation~\eqref{L2} appear for the first time in M.Sabatini paper \cite{S2}, when  the sufficient conditions of the isochronicity of the origin $O$  for the system~\eqref{L2P}  with $f$ and $g$ of classe $C^1$ are given.  

In the analytic case, the necessary and sufficient conditions for isochronicity are given  by  A.R.Chouikha in \cite{C3},
where the particular case of system  system~\eqref{L2P}
\begin{equation} \left.\begin{split}\dot x &= - y + b x^2y\\
 \dot y &= x + a_1x^2 + a_3y^2 + a_4x^3 + a_6xy^2 \end{split}  \right\} 
 \end{equation}
 is studied. In this system as well as in all other considered systems, all parameters are real.
 All the values of  the parameters for which the above system has the isochronous center at the origin are found.
 
In \cite{CRC} a similar result was obtained for more general system 
\begin{equation}\label{RCC}  \left. \begin{split} \dot x &= - y +a x y + b x^2y\\
 \dot y &= x + a_1x^2 + a_3y^2 + a_4x^3 + a_6xy^2 \end{split}\right\} 
 \end{equation}
 

The aim of this paper is  to extend investigations made in \cite{C3,CRC} for systems  with  higher order perturbations of the linear center $ \dot x=-y,\; \dot y=x$.  
We investigate the practical applicability and the limitations of the method developed in the cited papers for more complicated systems.   

 First let us consider the following particular case of~\eqref{L2P} which is more general then~\eqref{RCC}
 \begin{equation}\label{C_4} 
 \left. \begin{split}\dot x&= - y+ b_{1,1}yx+ b_{2,1}yx^2 + b_{3,1}yx^3\\
 \dot y&= x + a_{2,0}x^2 + a_{3,0}x^3+ a_{0,2}y^2 + a_{1,2}xy^2 + a_{2,2}x^2y^2+ a_{4,0}x^4 \end{split}\right\} 
\end{equation}
Because of computational complexity, we select for investigation two sub-families (first one $ b_{1,1}=a_{3,0}=0$, second one $b_{1,1}= b_{2,1}=0$) of the above system which have the codimension two in the parameter space.

In section~3, for the selected families we found all the parameters values for which the origin is an isochronous center. 
Thanks to this, among other, we found  three new additional isochronous cases of linear center perturbed by homogeneous polynomial,  which are not covered by the classification established by Chavarriga, Gin\'e and Garcia  in \cite{CGG}.
 For these three isochronous centers we give the explicit form of the first integral and the linearizing change of coordinates.
 
In the Section~4, an other particular case of the system~\eqref{L2P} is considered, namely the fifth degree homogeneous polynomial perturbation of linear center
\begin{equation}\label{CB5}\left.\begin{split}\dot x &= - y +ayx^4\\
 \dot y &= x +bx^3y^2+cx^5 \end{split}\right\} 
\end{equation}
We found all the parameters values for which the center at the origin is isochronous (two families).
The explicit form of the first integral and the linearizing change of coordinates are given for them. 
These systems are not contained in  the Chavarriga  et al. classification in \cite{CGG1}.

In the last section, we investigate the following particular Abel polynomial system 
\begin{equation}\label{AbI}
\left. \begin{split}
\dot x&=-y\\ \displaystyle\dot y&=\displaystyle\sum_{k=0}^{n} P_k (x){y}^k,
    \end{split}\right\} 
    \end{equation}
where  $P_k(x):= a_k x$  and $a_k \in \R$, for $k =0,\ldots, n$.  This Abel system 
 is also a particular case of~\eqref{L2P}, and  hence we can use the C-algorithm to investigate its isochronicity. 
 
 Volokitin and Ivanov \cite{VI} proved that for $n=3$ among systems of the form~\eqref{AbI}
  with arbitrary polynomials $P_k(x)\in\R[x]$, there is only one  family of isochronous centers. For $P_k(x)=a_kx$, this family reduces to exactly one system. Namely, the following one 
\begin{equation}\label{Ab31}
\left. \begin{split}
\dot x&=-y\\ \dot y&=x(1+y)^3
    \end{split}\right\} 
    \end{equation}
In the cited paper  Volokitin and Ivanov formulated the problem, which restricted to Abel equations of the form~\eqref{AbI}, can be stated as follows. 
Do exist among systems~\eqref{AbI} isochronous ones with $n\geq 4$? We give a partial negative answer to this question showing that for $4\leq n\leq 9$  among systems~\eqref{AbI} there is no an isochronous one.     

\section{Efficient algorithm for computing \\ 
 necessary conditions of isochronicity}
\subsection{About isochronous centers}

We collect now the results concerning Li\'enard type equation~\eqref{L2} (or its associated planar system~\eqref{L2P}) which will be used later.

Consider the  Li\'enard type equation 
\begin{equation*}
 \ddot x+f(x){\dot x}^2+g(x)=0, 
\end{equation*}
%
where $f$ and $g$ are $C^1$ class functions defined in a neighborhood $J$ of $0 \in \R$.  Let us  define the following  functions 
\begin{equation}
F(x):= \int_0^xf(s) ds, \quad \phi(x):= \int_0^x e^{F(s)} ds.
\end{equation}

When $xg(x)>0$ for $x\neq0$, define the function $X$  by
\begin{equation}
\frac {1}{2} X(x)^2 = \int_0^x g(s) e^{2F(s)} ds.
\end{equation}

\begin{theo}[H.Poincar\'e]
The planar system \eqref{1} with analytic data has an isochronous center at the origin if and only if for some analytic change of variables 
$u=u(x,y)=x+\dots,\; v=v(x,y)=y+\ldots$ the system \eqref{1} reduces to $\dot u=-kv,\;\dot v=ku$, where  $k\in\R,\;k\neq0$ and $\ldots$ denotes the higher order terms. 
\end{theo}
For more details see~\cite{MRT}.
\begin{theo}[Sabatini,\cite{S2}]\label{SAB}
Let $f ,\, g \in C^1(J,\R)$. If $xg(x)>0$ for $x\neq0$, then the system~\eqref{L2P} has a center at the origin $O$.
When $f,\; g$ are analytic , this condition is also necessary.

When $f ,\, g \in C^1(J,\R)$, the first integral of the system~\eqref{L2P} is given by the formula
 \begin{equation}
 I(x,\dot x)=2\int_0^x g(s) e^{2F(s)} ds+(\dot x e^{F(x)})^2
 \end{equation}
\end{theo}

\begin{theo}[Chouikha,\cite{C3}]
\label{thmC}
Let $f$, $g$ be  functions analytic in a neighborhood $J$ of $0$, 
and $x g(x) > 0$ for $x \neq 0$.
Then  system~\eqref{L2P}  has an isochronous center at $O$ if and only if 
there exists an odd function $h$ which satisfies the following conditions 
 \begin{equation}
 \label{CRI}\frac {X(x)}{1+h(X(x))} = g(x) e^{F(x)},
\end{equation}
the function $\phi(x)$ satisfies 
 \begin{equation}
 \label{bb} \phi(x) = X(x) + \int_0^{X(x)} h(t) dt,
\end{equation}
 and $X(x)\phi (x) > 0$ for $x\neq 0$.

In particular, when $f$ and  $g$ are odd, then $O$ is an isochronous center if and only if  $g(x) = e^{-F(x)}\phi (x)$, or equivalently  $h= 0$.
\end{theo}

The function $h$ is called {\it{Urabe function}}. As a corollaries of the above theorem one has
\begin{theo}[Chouikha,\cite{C3}]\label{COR}
Let $f$, $g$ be  functions analytic in a neighborhood of $0\in\R$, and $x g(x) > 0$ for $x \neq 0$.
If $g'(x)+g(x)f(x)=1$ then the origin $O$ is isochronous center of system~\eqref{L2P} and its associated Urabe funtion $h=0$. 
\end{theo}

\begin{theo}[Chouikha,\cite{C3}]\label{COR2}
Let $f$, $g$ be  functions analytic in a neighborhood $J$ of $0$, 
Consider the system~\eqref{L2P}  having a center at the origin $O$.
Let
  $$S(f,g) =  5g''^2(0) + 10g''(0)f(0) + 8f^2(0) - 3g'''(0) - 6f'(0).$$
 Then the following holds:\\
 (a)- $S(f,g) > 0$\ then the period function \ $T$\ increases in a neighborhood of $0$.\\
 (b)- $S(f,g) < 0$\ then the period function \ $T$\ decreases in a neighborhood of $0$.\\
 (c)- If~\eqref{L2P} has an isochronous center at $0$ then \  $S(f,g) = 0$.
 \end{theo}


\subsection{Chouikha algorithm }
The above Theorem~\ref{thmC} leads to an algorithm, first introduced by R. Chouikha in \cite{C3} (see also\cite{CRC}), in what follow called C-algorithm,
which allows to obtain necessary conditions for isochronicity
of the center at the origin for equation~\eqref{L2}.

Below we recall basic steps of the algorithm. 

Let $h$ be the function defined in the Theorem~\ref{thmC},  and $u ={\phi (x)}$.  We assume that function $\phi$ is invertible near the origin.
\begin{equation}
\label{tg}
	\tilde{g}(u) :=\frac{X}{1+h(X)},
\end{equation}
where now $X$ is considered as a function of $u$. Our further assumption is that  functions $f(x)$ and $g(x)$ depend polynomially on certain numbers of parameters $\alpha:=(\alpha_1,\ldots,\alpha_p)\in\R^p$.

By Theorem~\ref{thmC}, if the system~\eqref{L2} has isochronous center at the origin, then the function $h$ which is called the Urabe function, must be odd, so we have 
\begin{equation}
h(X)=\sum_{k=0}^{\infty}{c_{2k+1}} X^{2k+1},
\end{equation}
and moreover, 
\begin{equation}
\label{tgg}
\tilde{g}(u) = g(x)e^{F(x)}, \quad\text{where}\quad x=\phi^{-1}(u).	
\end{equation}
Hence, the right hand sides of~\eqref{tg} and~\eqref{tgg} must be equal. Hence, we expand the both right hand sides into the Taylor series around the origin and equate the corresponding coefficients. To this end we need to calculate $k$-th derivatives of~\eqref{tg} and~\eqref{tgg}.
  
For~\eqref{tg}, by straightforward differentiation, we have 
\begin{equation}\label{CR1}
\frac{d^k\tilde{g}(u)}{du^k} =\frac{d}{dX}\left(\frac{d^{k-1}\tilde{g}(u)}{du^{k-1}} \right) \frac{dX}{du}
\end{equation}
Using induction, one can show that for~\eqref{tgg} we obtain 
\begin{equation}
\label{CR2}\frac{d^k\tilde{g}(u)}{du^k} = e^{(1-k)F(x)}S_k(x),\end{equation}
where $S_k(x)$ is a function of $f(x), g(x)$ and their derivatives.

Therefore to compute the first $m$ conditions for isochronicity of system~\eqref{L2} we proceed as follows:
\begin{enumerate}
	\item we fix $m$ and write 
	\[
	h(X) =\sum_{k=1}^{m}{c_{2k-1}} X^{2k-1}+O(X^{2m}), \quad c:=(c_1, c_3, \ldots, c_{2m-1}),
\]
\item next, we compute 

	\[
	v_k :=\frac{d^k\tilde{g}}{du^k}(0), \quad w_k= S_k(0)
\]
 for $k = 1,\ldots, 2m + 1$. Note that those quantities are polynomials in $\alpha$ and $c$.
 \item by the Theorem~\ref{thmC} we obtain equations $v_k = w_k$ for $k = 1,\ldots, 2m + 1$. 
 
 It appears that we can always eliminate parameters $c$ from these equations. In effect we obtain a certain number of polynomial equations $s_1 = s_2 = s_3 = \ldots= s_M = 0$ with $p$ unknows $\alpha_i$. These equations gives $m$ necessary conditions for isochronicity of system~\eqref{L2}. 
\end{enumerate}
For more details see \cite{C3,CRC}.
\subsection{The choice of an appropriate Gr\"obner basis}

Let us consider the following system
 \begin{equation}\label{C_n} \qquad \left. \begin{array}{rl} \dot x&= - y+ b_{1,1}yx+\ldots+ b_{n-1,1}yx^{n-1}\\
 \dot y&= x + a_{2,0}x^2 +a_{0,2}y^2 +\ldots+ a_{n-2,2}x^{n-2}y^2+ a_{n,0}x^n\end{array}\right\} 
\end{equation}
which is reducible to the equation~\eqref{L2} with
\begin{equation}
	f(x) = \frac{a_{0,2}+b_{1,1} +\ldots+(a_{n-2,2}+(n-1)b_{n-1,1})x^{n-2}}{1-b_{1,1}x-\ldots- b_{n-1,1}x^{n-1}},
\end{equation}

\begin{equation}	
 g(x) = (x+ a_{2,0}x^2+\ldots + a_{n,0}x^n)(1-b_{1,1}x-\ldots- b_{n-1,1}x^{n-1}).
\end{equation}

In this paper we have investigated the two types of high degree polynomial perturbations, homogeneous and non-homogeneous ones.
 It seems that C-algorithm is efficient for computing isochronicity necessary conditions for higher degree homogeneous perturbations.
 In this case system~\eqref{C_n} reduces to the following one
 \begin{equation}\label{C_h} \qquad \left. \begin{array}{rl} \dot x&= - y+b_{k-1,1}yx^{k-1}\\
 \dot y&= x +a_{k-2,2}x^{k-2}y^2+ a_{k,0}x^k\end{array}\right\} 
\end{equation}
where $k\in \{ 2,\ldots,n\}$.

For this homogeneous polynomial perturbation of the linear center, C-algorithm generate homogeneous polynomial equations in the parameters 
$ a_{k-2,2},a_{k,0}$ and $b_{k-1,1}$.
Solving these polynomials, gives all the parameters values for which the real polynomial differential system~\eqref{C_h} is isochronous at the origin.

However, for the non-homogeneous perturbation case, C-algorithm generate non-homogeneous polynomial system.
Moreover, non-homogeneous perturbations depend on a bigger number of parameters.

We note that for $n=3$, C-algorithm succeeds to establish isochronicity criteria, however for $n=4$ the obtained polynomials from the algorithm are much more involved.
For example, for the system~\eqref{C_n} with $n=4$ reduces to
\begin{equation*}
 \left. \begin{split} \dot x&= - y+ b_{1,1}yx+ b_{2,1}yx^2 + b_{3,1}yx^3\\
 \dot y&= x + a_{2,0}x^2 + a_{3,0}x^3+ a_{0,2}y^2 + a_{1,2}xy^2 + a_{2,2}x^2y^2+ a_{4,0}x^4 \end{split}\right\} 
\end{equation*}
its associated first two non-zero  polynomials obtained by applying C-algorithm are the followings
\begin{equation}
\label{p3}
P_2=3\,b_{{2,1}}-3\,a_{{1,2}}+{b_{{1,1}}}^{2}-a_{{2,0}}b_{{1,1}}-9\,a_{{3,0
}}+4\,{a_{{0,2}}}^{2}-5\,b_{{1,1}}a_{{0,2}}+10\,{a_{{2,0}}}^{2}+10\,a_
{{2,0}}a_{{0,2}}, 
\end{equation}
\begin{equation*}
\begin{split}
P_3&=72\,{b_{{2,1}}}^{2}+396\,a_{{2,0}}b_{{1,1}}a_{{1,2}}+90\,b_{{1,1}}a_{{0
,2}}a_{{1,2}}+36\,b_{{1,1}}a_{{2,2}}+324\,b_{{3,1}}a_{{0,2}}\\
  & -36\,b_{{2
,1}}a_{{1,2}}-468\,a_{{2,0}}b_{{1,1}}b_{{2,1}}+612\,a_{{2,0}}b_{{2,1}}
a_{{0,2}}-4116\,b_{{1,1}}{a_{{2,0}}}^{2}a_{{0,2}}\\
& +108\,a_{{2,0}}b_{{3,
1}}
 -540\,a_{{3,0}}b_{{2,1}}-324\,a_{{4,0}}b_{{1,1}}+1566\,a_{{3,0}}b_{
{1,1}}a_{{0,2}}-288\,a_{{2,0}}a_{{2,2}}\\
&-459\,a_{{3,0}}{b_{{1,1}}}^{2}-
1296\,a_{{4,0}}a_{{0,2}}-306\,b_{{2,1}}b_{{1,1}}a_{{0,2}}+1428\,a_{{2,0
}}{b_{{1,1}}}^{2}a_{{0,2}}\\
&+153\,b_{{2,1}}{b_{{1,1}}}^{2}-117\,{b_{{1,1
}}}^{2}a_{{1,2}}-191\,a_{{2,0}}{b_{{1,1}}}^{3}+180
\,a_{{2,0}}a_{{0,2}}a_{{1,2}}
+43\,{b_{{1,1}}}^{4}\\
&-2319\,a_{{2,0}}b_{{1,1}}{a_{{0,2}}}^{2}-
289\,{b_{{1,1}}}^{3}a_{{0,2}}-360\,a_{{0,2}}a_{{2,2}}-36\,{a_{{1,2}}}^
{2}-171\,b_{{2,1}}{a_{{0,2}}}^{2}\\
&+513\,a_{{3,0}}{a_{{0,2}}}^{2}+537\,{
b_{{1,1}}}^{2}{a_{{0,2}}}^{2}+351\,{a_{{0,2}}}^{2}a_{{1,2}}-271\,b_{{1
,1}}{a_{{0,2}}}^{3}+542\,a_{{2,0}}{a_{{0,2}}}^{3}\\
&+756\,a_{{2,0}}a_{{3,0}}a_{{0,2}}+2268\,a_{{2,0}}a_{{3,0}}b_{{1,1}}-20\,{a_{{0,2}}}^{4}+1120
\,{a_{{2,0}}}^{4}+798\,{b_{{1,1}}}^{2}{a_{{2,0}}}^{2}\\
&-2240\,b_{{1,1}}{
a_{{2,0}}}^{3}-1512\,a_{{2,0}}a_{{4,0}}+1008\,{a_{{2,0}}}^{2}b_{{2,1}}
-252\,{a_{{2,0}}}^{2}a_{{1,2}}\\
&+1806\,{a_{{2,0}}}^{2}{a_{{0,2}}}^{2}+2240\,{a_{{2,0}}}^{3}a_{{0,2}}
\end{split}
\end{equation*}
To solve the first nine non-zero obtained polynomials requests high performance computer and the standard accessible computer algebra systems for solving polynomial equations are not able to find a solution. 

For  solving polynomial equations the  Gr\"obner bases are used. It is well known, see \cite{BW}, that the form and the size of the G\"obner basis of a polynomial ideal depends strongly on a  choice of monomial ordering. 
A bad choice of the monomial ordering can be a main reason why the Gr\"obner basis cannot be practically determined.

Our basic observation concerning algebraic structure of polynomial equations which give necessary conditions for the isochronicity in a case of non-homogeneous perturbations is following. Although the polynomials are not homogeneous, a careful analysis shows that they are quasi-homogeneous. 
In fact, one can notice that polynomial $P_2$ given by~\eqref{p3} is homogeneous if we give weight 2 for $b_{2,1}$, $a_{1,2}$ and $a_{3,0}$, and weight 1 for the remaining variables. More importantly we can find such a choice of weights for which all the obtained polynomials are homogeneous. 

The above observation shows that our main problem, i.e., finding a Gr\"obner basis, concerns as a matter of fact, finding a Gr\"obner basis of a homogeneous ideal. It is well know,  see  \cite[p.~466]{BW}, that homogeneous Gr\"obner bases have many 'nice' properties  which make them extremely useful  for solving  large and computationally demanding problems.

In fact, for non-homogeneous case of~\eqref{C_n}, the use of weighted degree gives a homogeneous Gr\"obner base.


To incorporate our observation into the C-algorithm we  choose a new parametrization for the problem.
First, we observe that all polynomials which are obtained by means of the C-algorithm are homogeneous if we choose the following weights 
\begin{enumerate}
	\item $i+j-1$ for parameters $a_{i,j}$ and $b_{i,j}$
	\item $2i+1$ for $c_{2i+1}$.
\end{enumerate}
Knowing this we introduced new parameters $A_{i,j}$, $B_{i,j}$, and $C_{2i+1}$ putting 
\begin{equation}
\label{ABC}
	A_{i,j}^{i+j-1}=a_{i,j}\quad B_{i,j}^{i+j-1}=b_{i,j}, \quad C_{2i+1}^{2i+1}=c_{2i+1}
\end{equation}
After this reparametrization system \eqref{C_n} reads
\begin{equation}\label{C_nh}
 \left. \begin{split} \dot x&= - y+ B_{1,1}yx+\ldots+ B_{n-1,1}^{n-1}yx^{n-1}\\
 \dot y&= x + A_{2,0}x^2 +A_{0,2}y^2 +\ldots+ A_{n-2,2}^{n-1}x^{n-2}y^2+ A_{n,0}^{n-1}x^n\end{split}\right\} 
\end{equation}
As in the case of isochronous center the Urabe function is odd, we search it under the form 
\begin{equation}
h(X)=\sum_{k=0}^{\infty}C_{2k+1}^{2k+1} X^{2k+1} = C_1 X + C_3^{3} X^3 +C_5^{5}X^5 +C_7^{7} X^7 +\ldots
\end{equation} 
We emphasize that from the isochronicity conditions for~\eqref{C_nh}, expressed in terms of its parameters,  it is easy to reconstruct the parameters values for which the system~\eqref{C_n} admits isochronous center at the origin, by a simply use of~\eqref{ABC}. 
%
%
%
%

Fact, that  the described reparametrization gives rise homogeneous equations, allows to reduce  the number of the parameters appearing in~\eqref{C_nh} by one. First, we assume  ${A_{2,0}}=0$, and then solve the isochronicity problem for system~\eqref{C_nh} under this assumption.
Next, for ${A_{2,0}}\neq0$, we apply on~\eqref{C_nh} the following change of coordinates
\begin{equation}
 (x,y) \mapsto\frac{1}{ A_{2,0}}( x,y)
\end{equation}
We  obtain
\begin{equation}
\left. 
\begin{array}{rl} \dot x &= - y +\left(\frac{B_{11}}{A_{2,0}}\right)xy +..+ \left(\frac{B_{n-1,1}}{A_{2,0}}\right)^{n-1}yx^{n-1}\\
 \dot y &= x + x^2 +\left( \frac{A_{0,2}}{A_{2,0}}\right)y^2  + ..+\left(\frac{A_{n-2,2}}{A_{2,0}}\right)^{n-1}x^{n-2}y^2+\left(\frac{A_{n,0}}{A_{2,0}}\right)^{n-1}x^n
 \end{array}\right\}
 \end{equation}
Hence, without loss of generality  we can put ${A_{2,0}}=1$, and find the parameters values for which the center is  isochronous. \\

We note that for an arbitrary  $k\in\N$, the problem of the isochonicity of the center for homogeneous perturbations  of the form~\eqref{C_h}  reduces to solve a number of polynomial equations in $3$ parameters.

 Recall that linear center perturbed by homogeneous polynomial, was investigated by W.S.~Loud in \cite{L} for the quadratic case, and in \cite{C3} the author find Loud results by the described algorithm, see also \cite{C1}.
 
Homogeneous perturbations was also studied by Chavarriga and coworkers.
For the fourth and fifth degree homogeneous perturbations, see \cite{CGG,CGG1},
where the homogeneous perturbations different from those studied in the present paper are considered.


%
%
%
%

\section{Fourth degree perturbations}

Let us consider the following system
 \begin{equation}\label{C_4} 
 \left. \begin{array}{rl} \dot x&= - y+ b_{1,1}yx+ b_{2,1}yx^2 + b_{3,1}yx^3\\
 \dot y&= x + a_{2,0}x^2 + a_{3,0}x^3+ a_{0,2}y^2 + a_{1,2}xy^2 + a_{2,2}x^2y^2+ a_{4,0}x^4 \end{array}\right\} 
\end{equation}
For the system~\eqref{C_4} having a center at the origin $O$, 
the following lemma gives a montonicity criteria for its  period function
\begin{lem}
Let
\begin{equation*}
S =10a_{2,0}^{2}+10a_{{2,0}}a_{{0,2}}-3a_{{1,2}}+3b_{{2,1}}-b_{{1,1}}a_{{2,0}}+b_{1,1}^{2}-9a_{{3,0}}+4a_{{0,2}}^{2}-5b_{{1,1}}a_{{0,2}}
\end{equation*}
 If $S > 0$ ($S < 0$) then the period function of~\eqref{C_4} is increasing (decreasing) at $O$.
 \end{lem}

\begin{proof}
System~\eqref{C_4} reduces to the Li\'enard type equation~\eqref{L2},
with
\begin{equation} 
 \left. \begin{split}f(x) &= \frac{a_{0,2}+b_{1,1} + (a_{1,2}+2b_{2,1})x +(a_{2,2} + 3b_{3,1})x^2}{1-b_{1,1}x-b_{2,1}x^2-b_{3,1}x^3}\\
 g(x) &= (x + a_{2,0} x^2+ a_{3,0} x^3 + a_{4,0}x^4)(1-b_{1,1}x-b_{2,1}x^2-b_{3,1}x^3)\end{split}\right\}
  \end{equation}
  
 then, we establish only three iterations (derivatives) of the C-algorithm given in Section~2, after elimination of $c_1$ of the Urabe function from the second derivatives, substitution in the third polynomial gives $S$.
Using the Theorem~\ref{COR2}, (See Corollary 2.8 of \cite{C3}), we prove the result.
\end{proof}

Because of computational complexity, we select for isochronicity investigation,  two sub-families of the system~\eqref{C_4}, which have the codimension two in the parameters space.

\subsection{First family}

Let us assume $b_{1,1}=a_{3,0}=0$, in this case~\eqref{C_4} reduces to the system

\begin{equation}\label{C4B}\left.\begin{array}{rl} \dot x& = - y + b_{2,1}x^2y +b_{3,1}yx^3\\
 \dot y &= x + a_{2,0}x^2 + a_{0,2}y^2  + a_{1,2}xy^2+a_{2,2}x^2y^2+a_{4,0}x^4 \end{array}\right\} \end{equation}
 
 For this system having a center at the origin $O$, we  give isochronicity necessary and sufficient conditions depending only on the following seven real parameters  $b_{2,1},b_{3,1}, a_{2,0}, a_{0,2}, a_{1,2},a_{2,2},a_{4,0}.$ 
 

\begin{theo} 
 The system~\eqref{C4B}  has an isochronous center at $O$ if and only if its parameters satisfy one of the folowing conditions
 \begin{description}
\item{1.} $a_{2,0}=a_{0,2}=b_{2,1}=a_{1,2}=0,\; b_{3,1}=-4 a_{4,0}/3,a_{2,2}=-16 a_{4,0}/3$
\item{2.} $ a_{2,0}=a_{0,2}=b_{2,1}= a_{1,2}=0, \; a_{4,0}=-b_{3,1}/2,\;a_{2,2}=b_{3,1}/2$
\item{3.} $ a_{2,0}=a_{4,0}=a_{0,2}=b_{2,1}=a_{1,2}=0,\; a_{2,2}=b_{3,1}$
\end{description}
which give the homogeneous perturbations and remaining ones which give the non-homogeneous :

\begin{description}
\item{4.} $ a_{2,0}=a_{4,0}=a_{0,2}=0,\; a_{2,2}=b_{3,1},\;b_{2,1}=a_{1,2}$
\item{5.} $a_{0,2}=-2a_{2,0},\; a_{1,2}=8a_{2,0}^2/3,\; a_{2,2}=-8a_{2,0}^3/3,$

$a_{4,0}=0, b_{3,1}=-4a_{2,0}^3/3,b_{2,1}=2a_{2,0}^2/3$

\item{6.} $a_{0,2}=-2a_{2,0},\; a_{1,2}=8a_{2,0}^2/3,\; a_{2,2}=-16a_{2,0}^3/21,$ 

$a_{4,0}=-4a_{2,0}^3/21,\;b_{3,1}=-4a_{2,0}^3/21,\;b_{2,1}=2a_{2,0}^2/3$

\item{7.} $a_{2,0}\neq0,\; a_{0,2}=-2a_{2,0},\; a_{4,0}=0,\;a_{2,2}=2 b_{3,1},$

$a_{1,2}=a_{2,0}^2(4+ b_{3,1}/a_{2,0}^3),\;  b_{2,1}=a_{2,0}^2(2+b_{3,1}/a_{2,0}^3)$
\end{description}
\end{theo}


In our investigations we have used {\it{Maple}} in its version 10. 
To compute the Gr\"obner basis of the obtained polynomial equations in the ring of characteristic $0$, we have used {\it{Salsa Software }} more precisely the {\it{FGb}} algorithm see \cite{F}.
 In this proof, we do not present the algorithm generated polynomials which are too long.
 \begin{proof}
We consider separately  the two cases : homogeneous  and non-homogeneous perturbations.

{\large Homogeneous perturbations}\\

In the homogeneous case system~\eqref{C4B} reduces to
\begin{equation}\label{H4}
\left.\begin{array}{rl} \dot x &= - y +b_{3,1}yx^3\\
 \dot y &= x +a_{2,2}x^2y^2+a_{4,0}x^4 \end{array}\right\} 
 \end{equation}
 
which is reducible to~\eqref{L2} such that

\begin{equation}
\left.\begin{split}
f(x) &= \frac {( 3 b_{3,1}+a_{2,2}) {x}^{2}}{1-b_{3,1}{x}^{3}}, \\
 g(x) &=( x-b_{3,1}{x}^{3} )(1+a_{4,0}{x}^{3})\end{split}\right\} 
 \end{equation}

 C-algorithm gives three homogeneous perturbations of linear center, candidate  to be isochronous.
 $19$ derivatives was essential to obtain the necessary conditions of isochronicity.
 We give explicitly the Urabe function $h$ associated to each system given by the algorithm.

Case~1: the system~\eqref{H4} becomes
\begin{equation}\label{S1}
\left.\begin{array}{rl} \dot x &= - y -\frac{4}{3}a_{4,0}yx^3\\
 \dot y &= x -\frac{16}{3} a_{4,0}x^2y^2+a_{4,0}x^4 \end{array}\right\}
 \end{equation}
 We can easily check that   $f(x)g(x) + g'(x)  =1 $, hence, following Corollary 2-7 of \cite{C3}, the system~\eqref{S1} is isochronous and $h(X)\equiv0.$\\
 
Case~2: the system~\eqref{H4} becomes

\begin{equation}\label{S2}  \left. \begin{split} \dot x &= - y +b_{3,1}yx^3 \\
 \dot y &= x +\frac{b_{3,1}}{2}x^2y^2-\frac{b_{3,1}}{2}x^4 \end{split}\right\} 
 \end{equation} 
The computations yield to the coefficients of the Urabe function :\\
\begin{equation*}\left. \begin{split}
c_1&=0,c_3=\frac{b_{3,1}}{2},c_5=0,c_7=0,\\ 
c_9&=-\frac{b_{3,1}^3}{16},c_{11}=0,c_{13}=0,c_{15}=\frac{3b_{3,1}^5}{256},\\
c_{17}&=0,c_{19}=0,c_{21}=-\frac{5b_{3,1}^7}{2048}. \end{split}\right.
 \end{equation*} 
Let $u_0=c_3,$ $u_1=c_9,$ $u_2=c_{15},$ $u_3=c_{21}\ldots$
We observe that for $k=1,2,3,4$ $$\frac{u_{k+1}}{u_k}=-\frac{b_{3,1}^2(k+\frac{1}{2})(k+1)}{4(k+1)^2}$$
It is then natural to conjecture that this is always the case.
 By series summation we found the odd  function 
 \begin{equation}\label{URABE}
 h(X)=\frac {b_{3,1} X^3}{\sqrt {4+{b_{3,1}}^{2}{X}^{6}}}=\sum_{k=1}^\infty u_k X^k
 \end{equation}
By direct computations we verify that the above $h$ satisfy the equation~\eqref{CRI}. Thus we conclude that the obtained $h$ is the Urabe function.

This case is similar to the one found in Theorem~3 of \cite{CRC}.\\

Case~3: the system~\eqref{H4} becomes

 \begin{equation}\label{S3} 
  \left. \begin{split} \dot x &= - y +a_{2,2}yx^3\\
 \dot y &= x +a_{2,2}x^2y^2 \end{split}\right\}
 \end{equation} 
 also this equation has an isochronous center at $O$ since 
  $f(x)g(x) + g'(x)  =1 $ and $ h(X)\equiv0.$\\
 
 We note that the case~3 is the case~4 with $a_{1,2}=0$.\\

{\large Non-homogeneous perturbations}\\

For the non-homogeneous perturbation of the linear center, we add to  C-algorithm the two tricks expanded in the last Section ( Homogenization and reduction of the dimension of the parameters space by 1 ).

We obtain,  the last four theorem cases~$4,\;5,\,6$ and $7$, when restricted to $a_{0,2}=1$,  satisfy   assumptions of Theorem~\ref{COR}, we show that each of the last four systems restricted to $a_{2,0}=1$ admit isochronous center at the origin and $ h=0$.
Finally, by the rescaling
\begin{equation}\label{RECI}
 (x,y) \mapsto a_{0,2}( x,y)
\end{equation}
we generalize the isochronicity for any $a_{2,0}\neq0$ since by change of coordinate~\eqref{RECI} the isochronicity is not lost.

\end{proof}

To complete the analysis, to each isochronous center  obtained by homogeneous perturnation, using Theorem~\ref{SAB} we write explicitely the first integrals and thanks to the method described in \cite{MRT} we compute the linearizing change of coordinates.


  A first integral of the system~\eqref{S1} is 
\begin{equation*}
 H_{\eqref{S1}}={\frac {\left(a_{4,0}^{2}{x}^{8}+2 a_{4,0}\,{x}^{5}+{x}^{2}+{y}^
{2}\right) ^3}{ 729\left( 3+4\, a_{4,0}\,{x}^{3} \right) ^{8}}}.\end{equation*}
The following analytic change of coordinates
\begin{equation*}
\left. \begin{split} u &={\frac {x \left( 1+ a_{4,0}\,{x}^{3} \right) }{3 \left( 3+4\, a_{4,0}\,{x}^{3} \right) ^{4/3}}}
\\
 v &={\frac {y}{ 3\left( 3+4\,{ a_{4,0}}\,{x}^{3} \right) ^{4/3}}} \end{split}\right\}\end{equation*}
 transforms the system~\eqref{S1} into the linear system :
 
\begin{equation*}\label{LIN}
\left. \begin{array}{rl} \dot u &= - v\\
 \dot v &= u \end{array}\right\}\end{equation*}
A first integral of system \eqref{S2} is
\begin{equation*}
H_{\eqref{S2}}={\frac {\left({x}^{2}+{y}^{2}\right)^3}{{b_{3,1}{x}^{3}-1}}},
\end{equation*}
then we give the linearizing change of coordinates 
 \begin{equation*}
\left. \begin{split} u &= {\frac {x}{\sqrt [6]{-1+b_{3,1}{x}^{3}}}}
\\
 v &={\frac {y}{\sqrt [6]{-1+ b_{3,1}{x}^{3}}}}\end{split}\right\}
 \end{equation*}
 A first integral of~\eqref{S3} is computed :
  \begin{equation*}
  H_{\eqref{S3}}={\frac {\left({x}^{2}+{y}^{2}\right)^3}{ \left( -1+{\it a_{2,2}}\,{x}^{3} \right) ^{2}}}
  \end{equation*}
  A linearizing change of coordinates is
  \begin{equation*}
\left. \begin{split} u &= {\frac {x}{\sqrt [3]{-1+{\it a_{2,2}}\,{x}^{3}}}}\\
 v &={\frac {y}{\sqrt [3]{-1+{\it a_{2,2}}\,{x}^{3}}}}\end{split}\right\}
 \end{equation*}

  For the remaining cases 3,4,5 and 6 obtained by non-homogeneous pertrbations, the first integrals and linearizing transformation can be obtained by the same way, but the computations are cumbersome and we dont report them there.
\subsection{Second family}

Consider system~\eqref{C_4}, with $b_{3,1}=b_{2,1}=0$.
We obtain the seven parameter real system of degree $4$.
\begin{equation}\label{C4B1}
\left. \begin{split} \dot x &= - y +b_{3,1}yx^3\\
 \dot y &= x + a_{2,0}x^2 + a_{0,2}y^2 + a_{3,0}x^3  + a_{1,2}xy^2+a_{2,2}x^2y^2+a_{4,0}x^4 \end{split}\right\} 
 \end{equation}
 For this system having a center at the origin $O$, we  give isochronicity necessary and sufficient conditions depending only on these seven real parameters  $a_{3,0},b_{3,1}, a_{2,0}, a_{0,2}, a_{1,2},a_{2,2},a_{4,0}$.
  
System~\eqref{C4B1} reduces to the equation~\eqref{L2}  with
\begin{equation*}
\left. \begin{split}f(x)& = \frac{a_{0,2} + a_{1,2} x +(a_{2,2} + 3b_{3,1})x^2}{1-b_{3,1}x^3},\\
 g(x)& = (x + a_{2,0} x^2+ a_{3,0} x^3  + a_{4,0}x^4)(1-b_{3,1}x^3)\end{split}\right\} 
 \end{equation*}
  

\bigskip
\begin{theo}
  The system~\eqref{C4B1}  has an isochronous center at $O$ if and only if its parameters satisfy one of the folowing conditions
\begin{description}
\item{1.} $a_{2,0}=a_{0,2}=b_{2,1}=a_{1,2}=0,\; b_{3,1}=-4 a_{4,0}/3,a_{2,2}=-16 a_{4,0}/3$
\item{2.} $ a_{2,0}=a_{0,2}=b_{2,1}= a_{1,2}=0, \; a_{4,0}=-b_{3,1}/2,\;a_{2,2}=b_{3,1}/2$
\item{3.} $ a_{2,0}=a_{4,0}=a_{0,2}=b_{2,1}=a_{1,2}=0,\; a_{2,2}=b_{3,1}$

\end{description}
which give the homogeneous perturbations and remaining ones which give the non-homogeneous
\begin{description}
\item{4.} $a_{2,0}=-a_{0,2}/2 ,\; a_{3,0}=a_{0,2}^2( 9-\sqrt {33})/48 ,\; a_{1,2}=a_{0,2}^{2}( -1+\sqrt {33})/{16}$ 
 
 $ a_{4,0}=0,\; a_{2,2}=a_{0,2}^{3}( -21+5\sqrt {33}) /{64},\; b_{3,1}=a_{0,2}^{3}( -21+5\,\sqrt {33} ) /192$
 
\item{5.} $a_{2,0}=-a_{0,2}/2 ,\; a_{3,0}={a_{0,2}}^2( 9+\sqrt {33})/{48} ,\; a_{1,2}=-a_{0,2}^{2}( 1+\sqrt {33})/16$
 
  $ a_{4,0}=0,\; a_{2,2}=-a_{0,2}^{3}( 21+5\sqrt {33}) /{64}, b_{3,1}=-a_{0,2}^{3}( 21+5\,\sqrt {33} ) /192$
  
\item{6.} $ a_{2,0}=-a_{0,2}/2,\;a_{3,0}=a_{4,0}=0,\;a_{1,2}=a_{0,2}^{2}/2$

$ a_{2,2}= a_{0,2}^{3}/2,\; b_{3,1}=a_{0,2}^{3}/4$

\item{7.} $a_{2,0}\neq0,\;a_{0,2}=-2a_{2,0},\; \frac{a_{3,0}}{{a_{2,0}}^2}={\frac {64}{117}}-{\frac {5717}{9975888}}\, \left( 22868+468\,\sqrt {3297} \right) ^{2/3}$

$+{\frac {1}{85264}}\,\sqrt {3297} \left( 22868+468
\,\sqrt {3297} \right) ^{2/3}
-{\frac {1}{117}}\,\sqrt [3]{22868+468\,
\sqrt {3297}},$

$\frac{a_{2,2}}{a_{2,0}^3}=-{\frac {344}{3549}}-{\frac {4720}{273}}\,{\frac {\sqrt {3297}}{ \left( 22868+468\,\sqrt {
3297} \right) ^{2/3}}}
+{\frac {96896}{3549}}\,{\frac {1}{\sqrt [3]{22868+468\,\sqrt {3297}}}}$

$+{\frac {32}{273}}\,{
\frac {\sqrt {3297}}{\sqrt [3]{22868+468\,\sqrt {3297}}}}-{\frac {
2695088}{3549}}\, \left( 22868+468\,\sqrt {3297} \right) ^{-2/3},
$

$a_{1,2}= a_{2,0}^2({\frac {-584+ \left( 22868+468\,\sqrt {3297} \right) ^{2/3}+14\,
\sqrt [3]{22868+468\,\sqrt {3297}}}{39\sqrt [3]{22868+468\,\sqrt {3297}}
}}),\;b_{3,1}=\frac{a_{2,2}}{4},$

$\frac{a_{4,0}}{a_{2,0}^3}={\frac {2150}{10647}}+{\frac {11926}{10647}}\,{\frac {1}{\sqrt [3]{22868+468\,\sqrt {3297}}}
}-{\frac {2}{91}}\,{\frac {\sqrt {3297}}{\sqrt [3]{22868+468\,\sqrt {
3297}}}}$

$-{\frac {1085248}{10647}}\, \left( 22868
+468\,\sqrt {3297} \right) ^{-2/3}-{\frac {160}{91}}\,{\frac {\sqrt {
3297}}{ \left( 22868+468\,\sqrt {3297} \right) ^{2/3}}}  $
 \end{description}
\end{theo}


\begin{proof}

 The same method used in the last proof, is employed.
 We establish the seven cases given in the theorem.
 
 The three first cases  belong to the family of linear centers perturbed by fourth degree homogeneous polynomial. Those are  analyzed in the proof concerned by the first family .
 
For each of the cases~4, 5 and 6 of the theorem restricted to $a_{0,2}=1$, it is easy to verify that conditions of Theorem \ref{COR} are satisfied, also for the case 7 restricted to $a_{2,0}=1$ we check easily $f(x)g(x) + g'(x)  =1 $.
Then, by an appropriate change of coordinate, we show that the system~\eqref{C4B1} satisfying one of the cases~4, 5, 6 and 7  have isochronous center at the origin.
\end{proof}

%

Note that the polynomials issued from the $19$ derivations and associated eliminations for the system~\eqref{C_4} (with 9 parameters), exceed the authorized memory of ordinary computers (2 Go of Random Access Memory ) in computations of the Gr\"obner basis by the known efficient algorithm FGb \cite{F}.
\section{Fifth degree homogeneous perturbations}
Let us consider the following system of degree $5$.
\begin{equation}\label{CB5}
\left. \begin{split} \dot x &= - y +ayx^4\\
 \dot y &= x +bx^3y^2+cx^5 \end{split}\right\}\end{equation}
 The systems given in the following theorem are additional isochronous cases to those established by Chavarriga et al in \cite{CGG1}.
\begin{theo} 
System~\eqref{CB5} has an isochronous center at the origin if and only if it reduces to one of the following systems
 \begin{equation}\label{51}
 \left. \begin{split} \dot x &= - y +ayx^4\\
 \dot y &= x + 5 ax^3y^2-\frac{4}{5} a x^5 \end{split}\right\} \end{equation}
 \begin{equation}\label{52}
 \left. \begin{split} \dot x&= - y +ayx^4\\
 \dot y &= x +a x^3y^2 \end{split}\right\}\end{equation}
\end{theo}
\begin{proof}
We perform the C-algorithm with the functions $f$ and $g$ such that
\begin{equation*}
 \left. \begin{split} 
  f(x)&=\frac{(b+4a)x^3}{(1-ax^4)}\\
  g(x)&=(x+cx^5)(1-ax^4)\end{split}\right\}
  \end{equation*}
then we obtain  isochronicity necessary conditions for the system~\eqref{CB5}.%
%
For the two  systems~\eqref{51}  and~\eqref{52}, it is easy to check that assymptions of Theorem \ref{COR} are satisfied
then systems~\eqref{51} and~\eqref{52} are isochronous and their associated Urabe function $ h=0.$
\end{proof}

As before, for each obtained isochronous center, we write explicitely the first integrals and the linearizing change of coordinates.

For the system~\eqref{51} a first integral is \begin{equation*}H_{\eqref{51}}={\frac {\left( \left( -5+4\,a{x}^{4} \right) ^{2}{x}^{2}+25\,{y}^{2}\right) ^{2}}{625
 \left( -1+a{x}^{4} \right) ^{5}}}
\end{equation*}
We give the linearizing change of coordinates for the system~\eqref{51}
\begin{equation*}
\left. \begin{split} u &={\frac { \left( -5+4\,a{x}^{4} \right) x}{5 \left( -1+a{x}^{4}
 \right) ^{5/4}}}
\\
 v &={\frac {y}{ \left( -1+a{x}^{4} \right) ^{5/4}}}
\end{split}\right\}\end{equation*}

For the system~\eqref{52} a first integral is 
\begin{equation*}
H_{\eqref{52}}={\frac {\left({x}^{2}+{y}^{2}\right)^2}{{-1+a{x}^{4}}}}.
\end{equation*}
We give the linearizing change of coordinates associated to~\eqref{52}
\begin{equation*}
\left. \begin{split} u &={\frac {x}{\sqrt [4]{-1+a{x}^{4}}}}\\
                        v &={\frac {y}{\sqrt [4]{-1+a{x}^{4}}}}
\end{split}\right\}\end{equation*}

\section{The period function for an Abel polynomial system}

This section is concerned with the following Abel system
\begin{equation}\label{Ab}
\left. \begin{split}
\dot x&=-y\\ \displaystyle\dot y&=\displaystyle\sum_{k=0}^{n} P_k (x){y}^k
    \end{split}\right\} 
    \end{equation}
where  $P_k(x):= a_k x$, $ a_0:=1$, and $a_k \in \R$, for $k=1,\ldots, n$.

  \subsection{Reduction to Li\'enard type system}
The system~\eqref{Ab} can be written

\begin{equation}\label{I}
\left.\begin{split} 
\dot x&=-y\\
 \dot y&=x(1+P(y)),
      \end{split}\right\} \end{equation}
with $ P(y)=a_1 y + a_2 y^2+a_3 y^3+....+a_n y^n.$  
 
Let define the functions $X$  and $\phi$ as follow
\begin{equation*}
\frac {1}{2} X(x)^2 =: \int_0^x \frac{s}{(1+P(s))} ds\quad\text{and}\quad\phi(x) =: \int_0^x\frac{ds}{(1+P(s))}.
\end{equation*} 
 We show the following.
\begin{theo}The origin $O$ is a center for~\eqref{I}.

Moreover, if $a_1^{2}-3 a_2 > 0$ $(a_1^{2}-3 a_2 < 0)$ then the period function of~\eqref{I} is increasing (decreasing ) at $O$.

 This center at $O$, is isochronous if and only if  there exists an odd function $h$  
satisfying
\begin{equation*}
\frac {X}{1+h(X)} = x
\end{equation*} 
and
\begin{equation*}\phi(x)= X(x) + \int_0^{X(x)} h(t) dt\end{equation*} 
such that  $X\phi(x) > 0$ for $x\neq 0$.\\
 
In particular, when $P$ is an even polynomial then the origin is isochronous center if and only if $P=0$
\footnote{The last statement of the above theorem was communicated to the autor by A.R.Chouikha.}.
 \end{theo}
\begin{proof}
From the symmetry criteria,  the origin is a center for~\eqref{I}, see \cite[chap.4]{C}. 
 We see that when $P(0)=0,$ there exists an open connected interval $J_1$
   containing $0$ where $1+P(y)\neq0$. Then  we can consider system in a neighborhood $U$ of the origin where  $U=J_1\times J_2$ with $J_2$ a suitable open interval containing $0$.
      So,  making the following transformation :
\begin{equation}\label{RES}
\left.\begin{split} 
y&=y\\z&=x(1+P(y))
      \end{split}\right\} \end{equation} 
we obtain, 
\begin{equation*}
\left.\begin{split} 
\dot y&=z\\ \dot z&=-y(1+P(y))+z^2 \frac{P'(y)}{(1+P(y))}
      \end{split}\right\}
       \end{equation*}
In effect, 
after renaming $y$ as $x$ and $z$ as $y$ we obtain the~\eqref{L2}.
The origin $O$ is a center for~\eqref{L2} with 
 \begin{equation*}
 \left. \begin{split}f(x)=-\frac{P'(x))}{(1+P(x))}\\g(x)=x(1+P(x))\end{split}\right\}
 \end{equation*}
For a fixed $n>2$, we establish  three iterations  of the C-algorithm given in Section~2, after elimination of $c_1$ of the Urabe function from the second derivation, substitution in the third polynomial gives  $a_1^{2}-3 a_2$.
By Using the Theorem~\ref{COR2} we prove the monotonicity result.

Let 
     \begin{equation*}
\left.\begin{split}F(x)& = \int_0^x f(s) ds=-\ln(1+P(x)),\\
  \phi(x)&=u=\int_0^x e^{F(s)} ds=\int_0^x\frac{ds}{(1+P(s))}\end{split}\right\} \end{equation*}
   Then we obtain
   \begin{equation*} \tilde g(u) = g(x) e^{F(x)}=x(1+P(x))e^{-\ln(1+P(x))}=x\end{equation*}  
  Following Theorem~\ref{thmC}, $\tilde g$ satisfies
   $$\tilde g(u) =  \frac {X}{1+h(X)}$$
   where $u = X + H(X)$.
    
   For the particular case when $P$ is even, it is easy to see that $f$ and $g$ are odd.
   We use the Theorem~\ref{thmC}, that leads us to obtain $x=X$ and $P\equiv0$.
   \end{proof}

The following paragraph is devoted to illustrate the last theorem by example.

 \subsection{Application to Volokitin and Ivanov system}

This section concerns the  Abel polynomial system 
\begin{equation}\label{ABN}
  \left. \begin{split}
\dot x&=-y\\ \dot y&=\sum_{i=0}^nP_i (x)y^i
    \end{split}\right\}\end{equation}
with $P_k\in\R[x]$, $0\leq k \leq n$.  

For such system the origin $O$ is not always a center.

\begin{theo}[Volokitin and Ivanov,~\cite{VI}]
The polynomial Abel system
\begin{equation}\label{VIG}
  \left. \begin{split}
\dot x&=-y\\ \dot y&=(x+a^2x^3)(1+h(x)y)^3+3axy(1+h(x)y)^2-h'(x)y^3
    \end{split}\right\}
    \end{equation}
    with an arbitrary number $a\in \R$ and an arbitrary polynomial $h\in\R[x]$ has an isochronous center at $O$.
\end{theo}

 Let  us consider the Abel system~\eqref{Ab} with $n=9$ :
\begin{equation}\label{AB9}
  \left. \begin{split}
\dot x&=-y\\ \dot y&=x+ \sum_{i=1}^9a_i x y^i
    \end{split}\right\}\end{equation}
with  $ a_k \in\mathbb R$, $1\leq k \leq 9$.
As follows from the symmetry criteria,  the origin $O$ is always a center for~\eqref{AB9}, see \cite[chap.4]{C}.
\begin{theo}\label{ISL}
Only in the case 

 \begin{equation}\label{I3}
\left. \begin{split}
\dot x&=-y\\ \dot y&=x+ a_1 xy+ \frac{a_1^2}{3} x{y}^2+ \frac{a_1^3}{27} x{y}^3
    \end{split}\right\} \end{equation} 
   the system~\eqref{AB9} has an isochronous center at the origin $O$.
   That is, only in the case when the system~\eqref{AB9} belongs to the Volokitin-Ivanov class~\eqref{VIG}.
\end{theo}
\begin{proof}
We perform C-algorithm  with
\begin{equation*}
\left. \begin{split}
f(x)&=-\frac{P'(x)}{(1+P(x))}\\ g(x)&=x(1+P(x))
    \end{split}\right\} \end{equation*} 
where $P(x)=\sum_{i=1}^9a_i x^i$.
We obtain the unique one-parameter family~\eqref{I3},
    and computations gives the Urabe function 
    \begin{equation*} h(X) = -\frac{a_1X}{3} =\frac{k_1 X}{\sqrt{{k_2}^2+k_3 X^2}}
    \end{equation*}
  with $k_1=-a_1/3,\; k_2=1,\; k_3=0$.
  
  It is interesting to note that this Urabe function is of the same nature that the Urabe function~\eqref{URABE} of the system~\eqref{S2}.
  (cf. also the proof of the Lemma~3.4 of \cite{C3} as well as Section~3 and 4 of \cite{CRC}.)
\end{proof}
We transform  system~\eqref{I3},  by the following change of coordinates 
 \begin{equation*}\left. \begin{split}
 \xi&=\frac{a_{1} x}{3}\\  \zeta &=\frac{a_{1} y}{3}.
    \end{split}\right\} \end{equation*} 
Then, after renaming $\xi$ as $x$ and $\zeta$ as $y$ we obtain the Volokitin and Ivanov system \cite[p.24]{VI}, which is a particular case of \eqref{VIG}
 \begin{equation}\label{Ab3}\left.\begin{split}
\dot x&=-y\\ \dot y &=x(1+y)^3.
    \end{split}\right\}  \end{equation}
In \cite{VI}, they prove that $O$ is an isochronous center of~\eqref{Ab3} showing that it commutes with some transversal polynomial system, but dont provide its first integral.

Thanks to our rescaling~\eqref{RES}, we determine a first integral of~\eqref{Ab3}
\begin{equation*} 
I_{\eqref{Ab3}}(x,y)=x^2+\frac{y^2}{(1+y)^2}\
\end{equation*}
 Indeed, let $P(x)=3 x+ 3  x^2+  x^3$. The system~\eqref{Ab3} reduces to the system~\eqref{L2P} with
\begin{equation*} 
\left.\begin{split}
f(x)&=-\frac{P'(x))}{(1+P(x))}\\
g(x)&=x(1+P(x)) .\end{split}\right\}
       \end{equation*}
 By Theorem \ref{SAB} we have
 \begin{equation*}
  I(x,\dot x)=2\int_0^x g(s) e^{2F(s)} ds+(\dot x e^{F(x)})^2.
    \end{equation*} 
is a first integral of~\eqref{L2P} which, in our case, reduces  to
\begin{equation*}
I(x,y)=\frac{x^2}{(1+x)^2}+\frac{y^2}{(1+x)^6}
\end{equation*}
By the reciprocal of the rescaling~\eqref{RES}, we obtain the first integral of the system~\eqref{Ab3}
\begin{equation*}
I_{\eqref{Ab3}}(x,y)=x^2+\frac{y^2}{(1+y)^2}.
\end{equation*}

Unfortunately we are unable to find explicitely the linerizing transformation for the system~\eqref{Ab3}.

In the light of Theorem~\ref{ISL},  it is natural to ask if the system~\eqref{I3} is the unique system with isochronous center at the origin inside the  family ~\eqref{Ab}.

Even for the system~\eqref{Ab} whith $n=10$, our actual computer possibilies are not sufficient to give an answer.


 \begin{center}
 {\it{Acknowledgments}} 
 \end{center}
 
 I would like to thank A.Raouf Chouikha  (University Paris 13, France) who introduced me to the subject,  proposed me the problem studied in this paper and also for many important hints and discussions.
 
 I thank  Jean Marie Strelcyn (University of Rouen, France) for many useful discussions and for help in the final redaction of this paper.
 
 
My thank's to Magalie Bardet (University of Rouen, France) who explained me the homogenization in practical use of the Gr\"obner basis method.

\end{document}